\documentclass[reqno,11pt]{amsart}
\usepackage{dsfont}
\usepackage{amsmath,amssymb}
\usepackage{mathrsfs}
\usepackage{paralist}
\usepackage{graphics} 
\usepackage{epsfig} 
\usepackage[colorlinks=true]{hyperref}
\hypersetup{urlcolor=red, citecolor=blue}
\usepackage[top=1in,bottom=1in,left=1in,right=1in]{geometry}
\usepackage{mathrsfs}
\newtheorem{thm}{Theorem}[section]

\newtheorem{lem}[thm]{Lemma}

\theoremstyle{definition}
\newtheorem{defn}{Definition}[section]
\newtheorem{rem}{Remark}[section]
\numberwithin{equation}{section}

\begin{document}
\title[]
      {Global existence and boundedness of weak solutions to a chemotaxis-stokes system with rotational flux term}%
\author[Li]{Li Feng}%
\address{Department of Mathematics, Southeast University, Nanjing 210096, P. R. China\\
}
\email{lilifeng1234567@163.com}


\author[Li]{Li Yuxiang}%
\address{Department of Mathematics, Southeast University, Nanjing 210096, P. R. China}
\email{lieyx@seu.edu.cn}
\thanks{Supported in part by National Natural Science Foundation of China (No. 11171063, No. 11671079, No. 11701290), National Natural Science Foundation of China under Grant (No. 11601127), and National Natural Science Foundation of Jiangsu Provience (No. BK20170896).}



\begin{abstract}

In this paper, the three-dimensional chemotaxis-stokes system
\begin{eqnarray*}
  \left\{\begin{array}{lll}
     \medskip
    n_{t}+u\cdot\nabla n=\Delta n^m-\nabla\cdot(n S(x,n,c)\cdot\nabla c),&x\in\Omega,\ \ t>0,\\
     \medskip
    c_t+u\cdot\nabla c=\Delta c-nf(c),&x\in\Omega,\ \ t>0,\\
     \medskip
     u_t+\nabla P=\Delta u +n\nabla\phi,&x\in\Omega,\ \ t>0,\\
     \nabla\cdot u=0, &x\in\Omega,\ \ t>0,,
  \end{array}\right.
\end{eqnarray*}
posed in a bounded domain $\Omega\subset\mathbb{R}^3$ with smooth boundary is considered under the no-flux boundary condition for $n$, $c$ and the Dirichlect boundary condition for $u$ under the assumption that the Frobenius norm of the tensor-valued chemotactic sensitivity $S(x,n,c)$ satisfies $S(x,n,c)<n^{l-2}\widetilde{S}(c)$ with $l>2$ for some non-decreasing function $\widetilde{S}\in C^{2}((0,\infty))$. In present work, it is shown that the weak solution is global in time and bounded while $m>m^\star(l)$, where
\begin{eqnarray*}
m^\star(l)=
\left\{\begin{array}{lll}
  \medskip
  l-\frac{5}{6},\ &\mathrm{if}\ \frac{31}{12}\geq l>2,\\
  \medskip
  \frac{7}{5}l-\frac{28}{15},\ &\mathrm{if}\ l>\frac{31}{12}.
\end{array}\right.
\end{eqnarray*}
\end{abstract}
\maketitle
\section{Introduction}

This paper deals with the global existence of weak solutions to the chemotaxis-stokes system with rotational flux
\begin{eqnarray}\label{the main system}
  \left\{\begin{array}{lll}
     \medskip
    n_{t}+u\cdot\nabla n=\Delta n^{m}-\nabla\cdot(n S(x,n,c)\cdot\nabla c),&x\in\Omega,\ \ t>0,\\
     \medskip
    c_t+u\cdot\nabla c=\Delta c-nf(c),&x\in\Omega,\ \ t>0,\\
     \medskip
     u_t+\nabla P=\Delta u +n\nabla\phi,&x\in\Omega,\ \ t>0,\\
    \nabla\cdot u=0, &x\in\Omega,\ \ t>0,
  \end{array}\right.
\end{eqnarray}
 in a bounded convex domain $\Omega\subset\mathbb{R}^{3}$. Here the chemotaxis sensitivity $S(x,n,c)$ is a matrix-valued function in $\mathbb{R}^{3\times 3}$ satisfying $|S(x,n,c)|\leq n^{l-2}\widetilde S(c)$ with $l>2$ and nondecreasing function $\widetilde S$ .

As described in \cite{Dombrowski2004Self}, the model was arisen to describe the behavior of swimming aerobic bacteria in situations where besides their chemotactically biased movement toward oxygen as their nutrient, a buoyancy-driven effect of bacterial mass on the fluid motion is not negligible. In the system (\ref{the main system}), density denoted by $n=n(x,t)$, affects the fluid motion, as represented by its velocity field $u=u(x,t)$ and the associated pressure $P=P(x,t)$, through buoyant forces. Moreover, it is assumed that both cells and oxygen, the latter with concentration $c=c(x,t)$, are transported by the fluid and diffuse randomly, that the cells partially direct their movement toward increasing concentration of oxygen, that the latter is consumed by the cells.
\vskip 2mm
\textbf{Keller-Segel model}. In 1970s, Keller and Segel in \cite{Keller1970Initiation} proposed a mathematical system
\begin{eqnarray}\label{classical keller segel model}
  \left\{\begin{array}{lll}
     \medskip
    n_{t}=\nabla\cdot(D(n)\nabla n)-\nabla\cdot(S(n)\nabla c),&x\in\Omega,\ \ t>0,\\
    c_t=\Delta c-c+ n,&x\in\Omega,\ \ t>0,
  \end{array}\right.
\end{eqnarray}
to indicate the collective motion of cells, usually bacteria or amoebae, that are attracted by a chemical substance and are able to emit it. For a general introduction to chemotaxis, see \cite{MR1995108}. In this model, $n$ is the density of a biological organism and $c$ is the concentration of a chemical substance produced by the  biological organism. The mathematical properties of (\ref{classical keller segel model}) have been extensively studied by many authors. The most peculiar features of (\ref{classical keller segel model}) are the existence, blow-up and asymptotic behavior to the non-radial or radial solutions under some suitable initial data $n_{0}(x)$. For instance, Ci$\acute{e}$slak and Winkler in \cite{MR3595313} have established the solution is global bounded in in a two-dimensional bounded domain under the assumption that $D(s)$ is decaying exponentially and $\frac{S(s)}{D(s)}\leq K s^\alpha$ is fulfilled with $\alpha\in(0,1)$. It has also been proven that there exist global bounded solutions when $S(n)\leq C(1+n)^{-\alpha}$ for all $s\geq 1$ and some $\alpha>1-\frac{2}{N}$ by Horstmannn and Winkler in \cite{MR2146345}, whereas the solutions may blow up in the radial case with $S(n)>cn^{-\alpha}$ for all $\alpha<1-\frac{2}{N}$ $(N\geq 2)$. Related models with prevention of overcrowding, see \cite{MR1826309}, volume effects \cite{MR2247456,MR2130723,MR2218162}, with logistic source \cite{MR3412307} or involving more than one chemo-attractant have also been studied in \cite{MR3304707,MR3247294,MR3198865}. Recently, the focus of the investigation to  chemotaxis model has extended to the chemotaxis-(Navier-)Stokes system.
\vskip 2mm
\textbf{Chemotaxis coupled with fluid}. Chemotaxis is an oriented immigration of living cells or bacteria under the effects of chemical gradients. Some aerobic bacteria such as Bacillus subtilis often live in thin fluid layer near solid-air-water contact, in which the swimming bacteria move towards higher concentration of oxygen according to mechanism of chemotaxis and meanwhile the movement of fluid is influenced by the gravitation force generated by itself. Thus, models with cell-fluid interaction become more complicated than fluid-free case as in (\ref{classical keller segel model}) since it does not only account for chemotaxis and diffusion, but also includes viscous fluid dynamics. Considering that the motion of the fluids is determined by the incompressible (Navier-)Stokes equations, the corresponding chemotaxis-fluid model was proposed as follows:
\begin{align}\label{chemotaxis with fluid}
\left\{\begin{array}{lll}
     \medskip
    n_{t}+u\cdot\nabla n=\nabla\cdot(D(n)\nabla n)-\nabla\cdot(S(x,n,c)\nabla c) ,&x\in\Omega,\ \ t>0,\\
     \medskip
    c_t+u\cdot\nabla c=\Delta c+f(n,c),&x\in\Omega,\ \ t>0,\\
     \medskip
     u_t+k(u\cdot\nabla)u+\nabla P=\Delta u +n\nabla\phi,&x\in\Omega,\ \ t>0,\\
    \nabla\cdot u=0, &x\in\Omega,\ \ t>0,
  \end{array}\right.
\end{align}
where $n$, $c$, $u$, $P$ and $\phi$ are defined as before, $k\in\mathbb{R}$ denotes the strength of nonlinear fluid convection. Moreover,
\begin{align}\label{f1}
f(n,c)=-c+n
\end{align}
when the production of cells or bacteria dominates the chemoattractant. Otherwise, if the signal is consumed, rather than produced, by the cells,  the function $f(n,c)$ is defined by
\begin{align}\label{f2}
f(n,c)=-ng(c)
\end{align}
for some known $g(c)$ denoting the consumption rate of the oxygen by the cells. From a viewpoint of mathematical analysis, this system couples the known obstacles from the theory of fluid equations to the typical  difficulties arising in the study of chemotaxis system. Despite this challenge, the well-posedness to the system (\ref{chemotaxis with fluid}) with (\ref{f1}) or (\ref{f2}) have been addressed under  various assumptions on the scalar function $S$, $f$ and $\phi$. For instance, Liu and Wang \cite{MR3566484} have proved that the solution to the system (\ref{chemotaxis with fluid}) is global in time and bounded for $k\neq 0$ and $N=3$ or $k=0$ and $N\in\{2,3\}$ under assumptions that (\ref{f1}) and
$$S(x,n,c)=\frac{\xi_0}{(1+\mu c)^2}$$
are satisfied.
When the $f(n,c)$ is defined by (\ref{f2}), Winkler \cite{MR3542616} has established the global existence of weak solution in a three-dimensional domain when $S(x,n,c)=D(n)\equiv 1$ and $k\neq 0$. Based on the method applied in the paper mentioned just before, Zhang and Li \cite{MR3369260} further shows the same result for $m>\frac{2}{3}$ when considering the porous media diffusion $D(n)=n^{m-1}$. If the $k=0$, the solutions to the system (\ref{chemotaxis with fluid}) coupled with (\ref{f2}) and porous media diffusion are bounded in time when $m>\frac{9}{8}$ and $S(x,n,c)\equiv 1$ are fulfilled. For more literatures related to this model, we can refer to \cite{MR2679718,MR3190352,MR2838394,MR2871341,MR3011296,MR3144997} and the reference therein.
\vskip 2mm
\textbf{Chemotaxis with rotational flux}.  Experiments find that the oriented migration of bacteria or cells may not be parallel to the gradient of the chemical substance, but may rather involve rotational flux components. This requires the sensitivity function $S$ in (\ref{chemotaxis with fluid}) to be a matrix possibly containing nontrivial off-diagonal entries (see \cite{MR2505083} and \cite{MR3294964} for detailed model derivation) such as appearing e.g. in the prototype
\begin{align}\label{trivial S}
S=\alpha\left(
\begin{array}{cc}
1 & 0\\
0 & 1\\
\end{array}
\right)
+\beta\left(
\begin{array}{cc}
0 & -1\\
1 & 0 \\
\end{array}
\right),\ \ \ \alpha>0,\ \ \beta\in\mathbb{R},
\end{align}
in the two-dimensional case. This generalization results in considerable mathematical difficulties due to the fact that chemotaxis systems with such rotational fluxes lose some energy structure, which has served as a key to the analysis for scalar-valued $S$. In \cite{MR3383312}, Winkler investigate the classical Keller-Segel model with tensor-value sensitivity
\begin{align*}
\left\{\begin{array}{lll}
     \medskip
    n_{t}=\Delta n-\nabla\cdot(nS(x,n,c)\cdot\nabla c),&x\in\Omega,\ \ t>0,\\
    c_t=\Delta c-nc,&x\in\Omega,\ \ t>0,
  \end{array}\right.
\end{align*}
and obtained a generalized solution under the assumption that $S(x,n,c)\leq CS_0(c)$ for some nondecreasing function $S_0$ in $[0,\infty)$. Thereafter, this kind of chemotaxis sensitivity is also applied to chemotaxis model coupled with fluid. For instance, under the assumption that $S(x,n,c)\leq C(1+n)^{-\alpha}$ and $k=0$ are satisfied, Wang and Xiang \cite{MR3401606,MR3542964} established the existence of global bounded  solutions to system (\ref{chemotaxis with fluid}) with (\ref{f1}) for arbitrary large initial data in a 2D and 3D bounded domain respectively when $D(n)\equiv 1$, Peng and Xiang \cite{MR3654908} further shows the same results with $m+2\alpha>2$ and $m>\frac{3}{4}$ when the porous media type diffusion $D(n)=mn^{m-1}$ is considered in a 3D bounded domain.

If the signal is consumed, the global classical solution and its large time behavior  under smallness assumption on $\|c_0\|_{L^\infty(\Omega)} $  in a bounded domain $\Omega\in\mathbb{R}^N$ $(N=\{2,3\})$ are also been obtained by Cao and Lankeit in \cite{MR3562314,MR3531759} when $k=0$ and $|S(x,n,c)|\leq CS_0(c)$ are fulfilled with some proper $S_0$, Winker \cite{MR3426095} establish the solutions to the system with porous media diffusion are also global bounded and converge to the integral equilibrium when $m>\frac{7}{6}$. For more related works to the system (\ref{chemotaxis with fluid}) , \cite{MR3302296,MR3426095} can be referred to.
\vskip 2mm
From the introduction to the chemotaxis system with tensor-valued sensitivity above, we can infer that the existing results are only focused on the case $S(x,n,c)\leq C$ as $n\rightarrow\infty$. Motivated by the work \cite{MR2235324}, it is meaningful for us to investigate the system (\ref{chemotaxis with fluid}) with nonlinear diffusion $D(n)=mn^{m-1}$ and tensor-valued chemotactic sensitivity $S(n,c)\leq Cn^{l-2}\widetilde{S}(c)$ for $l>2$.
In order to formulate our results, we specify the precise mathematical setting: we shall subsequently consider the system (\ref{the main system}) under the boundary conditions
\begin{align}\label{boundary conditions}
\nabla n\cdot\nu=\nabla c\cdot\nu=0,\ \ \ u=0,\ \ \ x\in\partial\Omega,\ t>0
\end{align}
and the initial conditions
\begin{align}\label{initial data}
n(x,0)=n_0(x),\ \ \ c(x,0)=c_0(x)\ \ \ u(x,0)=u_0(x),\ \ \ x\in\Omega
\end{align}
in a bounded convex domain $\Omega\subset\mathbb{R}^3$ with smooth boundary under the assumption that
\begin{eqnarray}\label{the initial data conditions}
\left\{\begin{array}{lll}
   n_{0}\in L^{\infty}(\Omega)\ \mathrm{is\ nonnegative},\\
   c_{0}\in W^{1,\infty}(\Omega),\ \mathrm{is\ nonnegative},\\
   u_{0}\in D(A^{\alpha}), \quad \alpha\in(\frac{3}{4}, 1).
\end{array}\right.
\end{eqnarray}
Moreover, we let
\begin{align}\label{condition 1 of S}
&|S(x,n,c)|:=\max_{i,j\in\{1,2\}}\{s_{ij}(x,n,c)\},\ \mathrm{with}\ s_{ij}\in C^{2}(\overline{\Omega}\times[0,\infty)\times [0,\infty)),
\end{align}
and
\begin{align}\label{condition 2 of S}
&|S(x,n,c)|:\leq Cn^{l-2}\widetilde{S}(c)\ {\rm with}\ l>2
\end{align}
for all $(x,n,c)\in\overline\Omega\times [0,\infty)\times [0,\infty)$, where $\widetilde S$ is a nondecreasing function on $[0,\infty)$.

As for the time independent gravitational potential $\phi$  and $f$ in (\ref{the main system}), we require that
\begin{align}
&\phi\in W^{1,\infty}(\overline\Omega)\label{phi condition},\\
&f\in C^{1}([0,\infty))\ {\rm is\ nonnegative}.\label{f condition}
\end{align}¡¡
Before stating our main result, let us briefly introduce the definition of weak solution to the system (\ref{the main system})
\begin{defn}\label{definition:definition of weak solution}
We say that $(n,c,u,P)$ is a global weak solution of (\ref{the main system}) associated to initial data $(n_0,c_0,u_0)$ if
\begin{align*}
&n\in L^{1}_{\mathrm{loc}}([0,\infty)\times\overline{\Omega}),\\
&c\in L^{\infty}_{\mathrm{loc}}(\overline{\Omega}\times[0,\infty))\cap L^{1}_{\mathrm{loc}}([0,\infty),W^{1,1}(\Omega)),\\
&u\in \left(L^{1}_{\mathrm{loc}}([0,\infty);W^{1,1}(\Omega))\right)^{3}
\end{align*}
such that $n\geq 0$ and $c\geq 0$ $\mathrm{a.e\ in}$ $\Omega\times [0,\infty)$,
\begin{align*}
&nc\in L^{1}_{\mathrm{loc}}([0,\infty)\times\Omega),\ \ \ u\otimes u\in \left(L^{1}_{\mathrm{loc}}([0,\infty)\times\Omega)\right)^{3\times 3},\ \ \ \mathrm{and}\\
&nS(x,n,c)\nabla c, nu, \mathrm{and}\ cu\ \mathrm{belong\ to}\ (L^{1}_{\mathrm{loc}}([0,\infty)\times\Omega))^{3},
\end{align*}
that $\nabla\cdot u=0$ a.e in $\Omega\times (0,\infty)$, and that
\begin{align}\label{identity 1 in weak sense}
-\int_{0}^{\infty}\int_{\Omega}n\varphi_{t}{\rm d}x{\rm d}t-\int_{\Omega}n_{0}\varphi(\cdot,0){\rm d}x=&-\int_{0}^{\infty}\int_{\Omega}\nabla n^{m}\cdot\nabla\varphi{\rm d}x{\rm d}t+\int_{0}^{\infty}\int_{\Omega}nS(x,n,c)\nabla c\cdot\nabla\varphi{\rm d}x{\rm d}t\notag\\&+\int_{0}^{\infty}\int_{\Omega}nu\cdot\nabla\varphi{\rm d}x{\rm d}t
\end{align}
for all $\varphi\in C_{0}^{\infty}([0,\infty)\times\overline{\Omega})$,
\begin{align}\label{identity 2 in weak sense}
-\int_{0}^{\infty}\int_{\Omega}c\varphi_{t}{\rm d}x{\rm d}t-\int_{\Omega}c_{0}\varphi(\cdot,0){\rm d}x=&-\int_{0}^{\infty}\int_{\Omega}\nabla c\cdot\nabla\varphi{\rm d}x{\rm d}t-\int_{0}^{\infty}\int_{\Omega}nf(c)\varphi{\rm d}x{\rm d}t\notag\\&+\int_{0}^{\infty}\int_{\Omega}cu\cdot\nabla\varphi{\rm d}x{\rm d}t
\end{align}
for all $\varphi\in C_{0}^{\infty}([0,\infty)\times\overline{\Omega})$ as well as
\begin{align}\label{identity 3 in weak sense}
-\int_{0}^{\infty}\int_{\Omega}u\cdot\zeta_{t}{\rm d}x{\rm d}t-\int_{\Omega}u_{0}\cdot\zeta(\cdot,0){\rm d}x=-\int_{0}^{\infty}\int_{\Omega}\nabla u\cdot\nabla\zeta{\rm d}x{\rm d}t+\int_{0}^{\infty}\int_{\Omega}n\nabla
\phi\cdot\zeta{\rm d}x{\rm d}t
\end{align}
hold for all $\zeta\in C_{0,\sigma}^{\infty}(\overline{\Omega}\times(0,\infty), \mathbb{R}^{3})$.
\end{defn}

\smallskip

Our main result reads as follows.
\begin{thm}\label{thm:the main theorem}
Let $\Omega\subset \mathbb{R}^3$ be a bounded convex domain with smooth boundary. Suppose that initial data and $S$ satisfies (\ref{the initial data conditions})-(\ref{condition 2 of S}) with some $m>m^\star(l)$, where
\begin{eqnarray}\label{definition of m star}
m^\star(l)=
\left\{\begin{array}{lll}
  \medskip
  l-\frac{5}{6},\ &\mathrm{if}\ \frac{31}{12}\geq l>2,\\
  \medskip
  \frac{7}{5}l-\frac{28}{15},\ &\mathrm{if}\ l>\frac{31}{12}.
\end{array}\right.
\end{eqnarray}
 Then, there exists a global weak solution $(n,c,u)$ to the system (\ref{the main system}) satisfying Definition \ref{definition:definition of weak solution}. And this solution is bounded in $\Omega\times(0,\infty)$ in the sense that with some $C>0$ we have
\begin{align}\label{boundedness result}
\|n(\cdot,t)\|_{L^\infty(\Omega)}+\|c(\cdot,t)\|_{W^{1,\infty}(\Omega)}+\|u(\cdot,t)\|_{W^{1,\infty}(\Omega)}\leq C\ \ \ \mathrm{for\ all}\ t>0.
\end{align}
Furthermore, $c$ and $u$ are continuous in $\overline{\Omega}\times[0,\infty)$ and
\begin{align}\label{weak continuous of n}
n\in C^{0}_{\omega-\star}([0,\infty); L^\infty(\Omega));
\end{align}
that is, $n$ is continuous on $[0,\infty)$ as an $L^\infty(\Omega)$-valued function with respect to the weak-$\star$ topology.
\end{thm}
\begin{rem}
If $l=2$ in the assumption (\ref{condition 2 of S}), the result above is consistent with \cite{MR3426095}.
\end{rem}

This paper is organized as follows. In Section 2, we approximate the problem by a well posed system (see(\ref{the regularized system}) later). Section 3 is devoted to study the boundedness of regularized problem, we will see that the bounds are independent of the way we regularized the problem. Thus, upon the appropriate uniform estimates, we can let $\varepsilon\rightarrow 0$ to obtain limit functions of the system (\ref{the regularized system}). This procedure is done in section 4, and also these limit functions are shown to solve (\ref{the main system}) in weak sense as defined in Definition \ref{definition:definition of weak solution}.
\vskip 3mm
\section{Approximation}
Our main methods is to apply the classical solution of an appropriately regularized system of (\ref{the main system}) to approximate the weak solution defined as in Definition \ref{definition:definition of weak solution}.
Following the same approximation procedure as in \cite{MR2505083}, we can find a family of functions $\{\rho_{\varepsilon}\}$ for any $\varepsilon\in(0,1)$ such that
\begin{align}\label{definiton of rho_varepsilon}
\rho_{\varepsilon}\in C_{0}^{\infty}(\Omega)\ \mathrm{with}\ 0\leq\rho_{\varepsilon}\leq 1\ \mathrm{in}\ \Omega\ \mathrm{and}\ \rho_{\varepsilon}\nearrow 1\ \mathrm{in}\ \Omega\ \mathrm{as}\ \varepsilon\searrow 0,
\end{align}
and define
\begin{align}\label{definition of S_varepsilon}
S_{\varepsilon}(x,n_{\varepsilon},c_{\varepsilon})=\rho_{\varepsilon}(x)S(x,n,c),\ x\in\overline{\Omega}.
\end{align}
Then, we have $S_{\varepsilon}(x,n_{\varepsilon},c_{\varepsilon})=0$ on $\partial\Omega$ and
\begin{align}\label{main property of S_varepsilon}
\left|S_{\varepsilon}(x,n_{\varepsilon},c_{\varepsilon})\right|\leq n_\varepsilon^{l-2}\widetilde S(c_{\varepsilon})\ \mathrm{for\ all}\ x\in\Omega, n_{\varepsilon}>0, c_{\varepsilon}>0.
\end{align}
Now, we consider the regularized system of (\ref{the main system}) as follows:
\begin{eqnarray}\label{the regularized system}
  \left\{\begin{array}{lll}
     \medskip
    n_{\varepsilon t}+u_{\varepsilon}\cdot\nabla n_\varepsilon=\Delta (n_\varepsilon+\varepsilon)^{m}-\nabla\cdot(n_{\varepsilon}S_{\varepsilon}(x,n_{\varepsilon},c_{\varepsilon})\cdot\nabla c_{\varepsilon}),&(x,t)\in\Omega\times (0, T),\\
     \medskip
    c_{\varepsilon t}+u_{\varepsilon}\cdot\nabla c_{\varepsilon}=\Delta c_{\varepsilon}-n_{\varepsilon}f(c_{\varepsilon}),&(x,t)\in\Omega\times (0, T),\\
     \medskip
     u_{\varepsilon t}+\nabla P_{\varepsilon}=\Delta u_{\varepsilon}+n_{\varepsilon}\nabla\phi,&(x,t)\in\Omega\times (0, T),\\
    \nabla\cdot u_{\varepsilon}=0, &(x,t)\in\Omega\times (0, T),\\
      \medskip
    \nabla n_{\varepsilon}\cdot\nu=\nabla c_{\varepsilon}\cdot\nu=0, u_{\varepsilon}=0, &(x,t)\in\partial\Omega\times (0,T),\\
      \medskip
    n_{\varepsilon}(x,0)=n_{0}(x), c_{\varepsilon}(x,0)=c_{0}(x), u_{\varepsilon}(x,0)=u_{0}(x), &x\in\Omega,
  \end{array}\right.
\end{eqnarray}

The first lemma concerns the local solvability of the regularized system (\ref{the regularized system}) in classical sense. Without lose of generality, the proof is based on well-established methods involving the Schaulder fixed point theorem, the standard regularity theory of parabolic equations and the Stokes system. For more details, we can refer to \cite{MR3383312}.

\begin{lem}\label{local solvability of the regularized system}
For all $\varepsilon\in(0,1)$, let $\Omega\subset\mathbb{R}^3$ be a bounded domain with smooth boundary. Assume that initial data $(n_{0}, c_{0}, u_{0})$ satisfies (\ref{the initial data conditions}), and $S$ fulfills (\ref{condition 1 of S})-(\ref{condition 2 of S}). Then there exist a maximal existence time $T_{max, \varepsilon}\in(0,\infty]$ and functions
\begin{align*}
&n_{\varepsilon}\in C^{0}(\overline{\Omega}\times [0,T_{max,\varepsilon}))\cap C^{2,1}(\overline{\Omega}\times (0,T_{max,\varepsilon})),\\
&c_{\varepsilon}\in C^{0}(\overline{\Omega}\times [0,T_{max,\varepsilon}))\cap C^{2,1}(\overline{\Omega}\times (0,T_{max,\varepsilon})),\\
&u_{\varepsilon}\in C^{0}(\overline{\Omega}\times [0,T_{max,\varepsilon}))\cap C^{2,1}(\overline{\Omega}\times (0,T_{max,\varepsilon})),\\
&P_\varepsilon\in C^1(\overline{\Omega}\times [0,T_{max,\varepsilon})),
\end{align*}
 such that $(n_{\varepsilon},c_{\varepsilon},u_{\varepsilon},P_\varepsilon)$ is a classical solution of (\ref{the regularized system}) in $\Omega\times (0,T_{max, \varepsilon})$, and such that $n_\varepsilon$ and $c_\varepsilon$ are nonnegative. Moreover, if $T_{max, \varepsilon}<\infty$,
\begin{align}\label{principle for global existence of solution}
\|n_{\varepsilon}(\cdot,t)\|_{L^{\infty}(\Omega)}+\|c_{\varepsilon}(\cdot,t)\|_{W^{1,\infty}(\Omega)}+\|A^\alpha u_{\varepsilon}(\cdot,t)\|_{L^2(\Omega)}\rightarrow\infty
\end{align}
as $t\rightarrow T_{max,\varepsilon}$, where $\alpha$ is defined in (\ref{the initial data conditions}).
\end{lem}

Therefore, in order to prove the global existence of the regularized problem, it is sufficient for us to show the boundedness for each term in (\ref{principle for global existence of solution}) under the assumption that $T_{max,\varepsilon}<\infty$. The following lemma is immediately obtain upon observation.

\begin{lem}\label{property 1 of n varepsilon and c varepsilon}
For any $\varepsilon\in(0,1)$, let $(n_\varepsilon,c_\varepsilon,u_\varepsilon,P_\varepsilon)$ be a classical solution to (\ref{the regularized system}). It follows that
\begin{align}\label{mass conservation of n varepsilon}
\int_{\Omega}n_{\varepsilon}(\cdot,t){\rm d}x=\int_{\Omega}n_{0}{\rm d}x,\ \ \mathrm{for\ all}\ t\in(0,T_{max,\varepsilon}),
\end{align}
and
\begin{align}\label{L^infty boundedness of c varepsilon}
\|c_{\varepsilon}(\cdot, t)\|_{L^{\infty}(\Omega)}\leq \|c_{0}\|_{L^{\infty}(\Omega)}\ \ \mathrm{for\ all}\ t\in(0,T_{max,\varepsilon}).
\end{align}
\end{lem}

\begin{proof}
The mass conservation (\ref{mass conservation of n varepsilon}) is a straightforward consequence of an integration of the first equation of (\ref{the regularized system}) over $\Omega$. Since $n_{\varepsilon}$ and $c_{\varepsilon}$ are nonnegative, an application of the maximum principle to the second equation in regularized system yields (\ref{L^infty boundedness of c varepsilon}).
\end{proof}

\vskip 3mm
\begin{lem}\label{relationships between boundedness of n varepsilon and Du varepsilon}
Let $p\in [1,\infty)$ and $r\in [1,\infty]$ be such that
\begin{eqnarray*}
  \left\{\begin{array}{lll}
  r<\frac{3p}{3-p},\ \ \ &\mathrm{if}\ p\leq 3,\\
  r\leq\infty,\ \ \ \ &\mathrm{if}\ p>3.
\end{array}\right.
\end{eqnarray*}
Then for all $K>0$ there exists $C_1=C(p,r,K)$ such that for some $\varepsilon\in(0,1)$ we have

\begin{align*}
\|n_\varepsilon(\cdot,t)\|_{L^p(\Omega)}\leq K\ \ \ \mathrm{for\ all}\ t\in(0,T_{max,\varepsilon}),
\end{align*}
then
\begin{align}\label{conclusion 1 to Du varepsilon}
\|Du_\varepsilon(\cdot,t)\|_{L^r(\Omega)}\leq C_1\ \ \mathrm{for\ all}\ t\in(0,T_{max,\varepsilon}).
\end{align}
\end{lem}
\begin{proof}
The proof of this lemma is based on some fundamental estimates of semigroup, and the details can be seen in \cite{MR3426095}  which is omitted here.
\end{proof}
\begin{lem}\label{boundedness of n varepsilon in L^p}
 Let $p>1$, (\ref{condition 2 of S}) is satisfied. Then for each $\varepsilon\in(0,1)$ we have
\begin{align}
\frac{1}{p}\frac{{\rm d}}{{\rm d}t}\int_{\Omega}(n_\varepsilon+\varepsilon)^p{\rm d}x&+\frac{2m(p-1)}{(m+p-1)^2}\int_{\Omega}|\nabla (n_\varepsilon+\varepsilon)^{\frac{m+p-1}{2}}|^2{\rm d}x\notag\\&\leq\frac{C_{0}^{2}(p-1)}{2m}\int_{\Omega}(n_\varepsilon+\varepsilon)^{p+2l-m-3}|\nabla c_{\varepsilon}|^{2}{\rm d}x\ \ \mathrm{for\ all}\ t\in(0,T_{max,\varepsilon}),
\end{align}
where
$$C_{0}=\widetilde{S}(\|c_0\|_{L^\infty(\Omega)})$$
with $\widetilde{S}$ defined as in (\ref{condition 2 of S}).
\end{lem}
 \begin{proof}
 We multiply the first equation in (\ref{the regularized system}) with $(n_\varepsilon+\varepsilon)^{p-1}$ $(p>1)$ and integrate by parts over $\Omega$. Since $\nabla u_\varepsilon$ and $S_\varepsilon(x,n_\varepsilon,c_\varepsilon)$ vanish whenever $x\in\partial\Omega$, this yields
 \begin{align*}
\frac{1}{p}\frac{{\rm d}}{{\rm d}t}\int_{\Omega}(n_\varepsilon+\varepsilon)^p{\rm d}x=&-\int_{\Omega}\nabla(n_\varepsilon+\varepsilon)^m\cdot\nabla(n_\varepsilon+\varepsilon)^{p-1}{\rm d}x\\&+\int_{\Omega}(n_\varepsilon S_\varepsilon(x,n_\varepsilon,c_\varepsilon)\nabla c_\varepsilon)\cdot\nabla(n_\varepsilon+\varepsilon)^{p-1}{\rm d}x\\
=& -m(p-1)\int_\Omega(n_\varepsilon+\varepsilon)^{m+p-3}|\nabla(n_\varepsilon+\varepsilon)|^2{\rm d}x\\&+(p-1)\int_\Omega n_\varepsilon(n_\varepsilon+\varepsilon)^{p-2}S_\varepsilon(x,n_\varepsilon,c_\varepsilon)\cdot\nabla(n_\varepsilon+\varepsilon){\rm d}x
\end{align*}
for all $t\in(0,T_{max,\varepsilon})$. Define
  $$C_{0}=\widetilde{S}(\|c_0\|_{L^\infty(\Omega)}).$$
 Due to the nonnegativity of $n_\varepsilon$, (\ref{main property of S_varepsilon}) and (\ref{L^infty boundedness of c varepsilon}), we have $|S_\varepsilon(x,n_\varepsilon,c_\varepsilon)|\leq C_{0}n_\varepsilon^{l-2}$. Furthermore, by applying Young's inequality, we derive that
\begin{align*}
\frac{1}{p}&\frac{{\rm d}}{{\rm d}t}\int_{\Omega}(n_\varepsilon+\varepsilon)^p{\rm d}x\\
&\leq -\frac{4m(p-1)}{(m+p-1)^2}\int_{\Omega}|\nabla (n_\varepsilon+\varepsilon)^{\frac{m+p-1}{2}}|^2{\rm d}x+(p-1)C_{0}\int_\Omega (n_\varepsilon+\varepsilon)^{p+l-3}\nabla c_\varepsilon\cdot\nabla(n_\varepsilon+\varepsilon){\rm d}x\notag\\&\leq -\frac{2m(p-1)}{(m+p-1)^2}\int_{\Omega}|\nabla (n_\varepsilon+\varepsilon)^{\frac{m+p-1}{2}}|^2{\rm d}x+\frac{C_{0}^{2}(p-1)}{2m}\int_{\Omega}(n_\varepsilon+\varepsilon)^{p+2l-m-3}|\nabla c_{\varepsilon}|^{2}{\rm d}x
\end{align*}
for all $t\in(0,T_{max,\varepsilon})$.
Thus, the proof of this lemma is completed.
 \end{proof}
\vskip 2mm
\begin{lem}\label{difference of nabla c varepsilon}
Let $q>1$ and $\varepsilon\in(0,1)$. Then we have
\begin{align}\label{differential inequality about nabla c 2q}
\frac{1}{2q}&\frac{{\rm d}}{{\rm d}t}\int_\Omega|\nabla c_\varepsilon|^{2q}{\rm d}x +\frac{3(q-1)}{8}\int_\Omega|\nabla c_\varepsilon|^{2(q-2)}\left|\nabla|\nabla c_\varepsilon|^2\right|^2{\rm d}x+\frac{1}{4}\int_{\Omega}|\nabla c_\varepsilon|^{2(q-1)}|D^2 c_\varepsilon|^2{\rm d}x\notag\\
&\leq(2-3+\sqrt{3})^2f_1^2\int_\Omega n_\varepsilon^2|\nabla c_\varepsilon|^{2q-2}{\rm d}x+(2q+1)\int_\Omega|\nabla c_\varepsilon|^{2q}|Du_\varepsilon|{\rm d}x
\end{align}
for all $t\in(0,T_{max,\varepsilon})$, where
\begin{align*}
f_{1}=\|f\|_{L^\infty(0,\|c_{0}\|_{L^\infty(\Omega)})}.
\end{align*}
\end{lem}
\begin{proof}
The omitted detail computation of this lemma can be seen in \cite{MR3426095}.
\end{proof}
\vskip 2mm
\begin{lem}\label{lem:estimate to u varepsilon}
For any $\varepsilon\in(0,1)$. Let $m\geq 1$ and suppose that $p>1$ satisfies
\begin{align}\label{p and m relation}
p>\frac{7-3m}{3}.
\end{align}
Then, we have
\begin{align}\label{estimate to u varepsilon}
\frac{{\rm d}}{{\rm d}t}\int_\Omega|A^{\frac{1}{2}}u_\varepsilon|^2{\rm d}x +\int_\Omega|Au_\varepsilon|^2{\rm d}x\leq\eta\int_\Omega|\nabla n_\varepsilon^{\frac{m+p-1}{2}}|^2{\rm d}x+C_2\ \ \ \mathrm{for\ all}\ t\in(0,T_{max,\varepsilon}),
\end{align}
for each $\eta>0$ and some positive constant $C_2=(p,m,\eta)$.
\end{lem}
\begin{proof}
We apply the $\mathrm{Holmholz}$ projection $\mathscr{P}$ to the third equation in (\ref{the regularized system}),  and then test the resulting equation
$$\partial_tu_\varepsilon+Au_\varepsilon=\mathscr{P}(n\nabla\phi)$$
with $Au_\varepsilon$ to obtain
\begin{align}\label{estimate 1 to u varepsilon}
\frac{1}{2}\frac{{\rm d}}{{\rm d}t}\int_\Omega|A^{\frac{1}{2}}u_\varepsilon|^2{\rm d}x+\int_\Omega|Au_\varepsilon|^2{\rm d}x&=\int_\Omega Au_\varepsilon\cdot\mathscr{P}(n_\varepsilon\nabla\phi){\rm d}x\notag\\
&\leq\frac{1}{2}\int_\Omega|Au_\varepsilon|^2{\rm d}x+\frac{1}{2}\int_\Omega|(n_\varepsilon\nabla\phi)|^2{\rm d}x\ \ \ \mathrm{for\ all}\ t\in(0,T_{max,\varepsilon}).
\end{align}
Furthermore, the Gagliardo-Nirenberg inequality yields that
\begin{align}\label{estimate 2 to u varepsilon}
\int_\Omega|(n_\varepsilon\phi)|^2{\rm d}x\leq&\|\phi\|_{W^{1,\infty}(\Omega)}^{2}\left\|(n_\varepsilon+\varepsilon)^{\frac{m+p-1}{2}}\right\|_{L^{\frac{4}{m+p-1}}(\Omega)}^{\frac{4}{m+p-1}}\notag\\
\leq &C_3\|\phi\|_{W^{1,\infty}(\Omega)}^{2}\|\nabla n_\varepsilon^{\frac{m+p-1}{2}}\|_{L^2(\Omega)}^\frac{6}{3m+3p-4}\cdot\|(n_\varepsilon+\varepsilon)^{\frac{m+p-1}{2}}\|_{L^{\frac{2}{m+p-1}}(\Omega)}
^{\frac{3(3m+3p-5)}{(3m+3p-4)(m+p-1)}}\notag\\
&+\|(n_\varepsilon+\varepsilon)^{\frac{m+p-1}{2}}\|_{L^{\frac{2}{m+p-1}}(\Omega)}^{\frac{4}{m+p-1}}\notag\\
\leq&C_4\|\nabla n_\varepsilon^{\frac{m+p-1}{2}}\|_{L^2(\Omega)}^\frac{6}{3m+3p-4}+C_4\ \ \ \mathrm{for\ all}\ t\in(0,T_{max,\varepsilon}),
\end{align}
is fulfilled for some $C_3=C(p,\eta)>0$ and $C_4=C(p,\eta)>0$ due to the (\ref{mass conservation of n varepsilon}).

Additionally, by the assumption (\ref{p and m relation}), we can see that
\begin{align*}
\frac{4}{m+p-1}<6,
\end{align*}
and
\begin{align*}
\frac{6}{3m+3p-4}<2
\end{align*}
are valid. Thus, the claimed inequality (\ref{estimate to u varepsilon}) results from (\ref{estimate 1 to u varepsilon}) and (\ref{estimate 2 to u varepsilon}) by a second application of Young's inequality.
\end{proof}
\vskip 2mm
We next plan to estimate the right-hand sides in the above inequality appropriately by using suitable interpolation arguments along with the basic priori information provided by Lemma \ref{property 1 of n varepsilon and c varepsilon}. Here, we introduce an auxiliary interpolation lemma, which will play an important role in making efficient use of the known $L^\infty$ bound for $c_\varepsilon$.
\vskip2mm
\begin{lem}\label{interpolation lemma about varphi}
Suppose that $\Omega\subset\mathbb{R}^3$ is a bounded domain with smooth boundary, that $q\geq1$ and that
\begin{align}\label{lambda condition}
\lambda\in[2q+2,4q+1].
\end{align}
Then there exists $C_5>0$ such that for all $\varphi\in C^2(\overline{\Omega})$ fulfilling $\varphi\cdot\frac{\partial\varphi}{\partial\nu}=0$ on $\partial\Omega$ we have
\begin{align}\label{interpolation inequality about varphi}
\|\nabla\varphi\|_{L^\lambda(\Omega)}\leq C_5\left\||\nabla\varphi|^{q-1}D^2\varphi\right\|_{L^2(\Omega)}^{\frac{2\lambda-6}{(2q-1)\lambda}}\|\varphi\|_{L^\infty(\Omega)}
^{\frac{2\lambda-6}{(2q-1)\lambda}}+C_5\|\varphi\|_{L^\infty(\Omega)}.
\end{align}
\end{lem}
\begin{proof}
The proof of this lemma can be found in \cite{MR3426095} for details.
\end{proof}
\vskip 2mm
The term on the right-hand of Lemma \ref{boundedness of n varepsilon in L^p} can be estimated as follows.
\begin{lem}\label{lem:upper bounds p}
Let $m\geq l-1$, $q>1$. Assume that $(n_\varepsilon, c_\varepsilon, u_\varepsilon,P_\varepsilon)$ is a solution to (\ref{the regularized system}) on $[0,T_{max,\varepsilon})$. Suppose that $p>\max\{1,m-2l+3\}$ satisfies
\begin{align}\label{upper bound of p}
p<\left(2m-2l+\frac{8}{3}\right)q+m-2l+3.
\end{align}
Then for all $\eta>0$ there exists $C_6=C(p,q,\eta)>0$ with the property that for all $\varepsilon\in(0,1)$,
\begin{align}\label{estimate 4}
\int_{\Omega}(n_\varepsilon+\varepsilon)^{p+2l-m-3}|\nabla c_{\varepsilon}|^{2}{\rm d}x\leq\eta\int_\Omega|\nabla(n_\varepsilon+\varepsilon)^{\frac{m+p-1}{2}}|^2{\rm d}x+\eta\int_{\Omega}|\nabla c_\varepsilon|^{2q-2}|D^2 c_\varepsilon|^2{\rm d} x +C_6
\end{align}
for all $t\in(0,T_{max,\varepsilon})$
\end{lem}
\begin{proof}
For all $t\in(0,T_{max,\varepsilon})$, we apply the $\mathrm{H\ddot{o}lder}$ inequality with exponents $\frac{q+1}{q}$ and $q+1$ to obtain
\begin{align}\label{estimate 5}
\int_\Omega n_\varepsilon^{p+2l-m-3}|\nabla c_\varepsilon|^2{\rm d}x&\leq\left(\int_\Omega n_\varepsilon^{\frac{(p+2l-m-3)(q+1)}{q}}{\rm d}x\right)^{\frac{q}{q+1}}\left(\int_\Omega|\nabla c_\varepsilon|^{2q+2}{\rm d}x\right)^{\frac{1}{q+1}}\notag\\
&=\left\|n_\varepsilon^{\frac{m+p-1}{2}}\right\|_{L^{\frac{2(p+2l-m-3)(q+1)}{(p+m-1)q}}(\Omega)}^{\frac{2(p+2l-m-3)}{p+m-1}}\cdot\left\|\nabla c_{\varepsilon}\right\|_{L^{2q+2}(\Omega)}^{2}.
\end{align}
By applying the Gagliardo-Nirenberg inequality and the mass conservation of $n_\varepsilon$ (\ref{mass conservation of n varepsilon}), there exist positive constants $C_7=C(p,q)$ and $C_8=C(p,q)$ satisfying
\begin{align}\label{estimate 5 the first term}
\left\|n_\varepsilon^{\frac{m+p-1}{2}}\right\|_{L^{\frac{2(p+2l-m-3)(q+1)}{(p+m-1)q}}(\Omega)}^{\frac{2(p+2l-m-3)}{p+m-1}}\leq & C_7\left\|\nabla n_\varepsilon^{\frac{m+p-1}{2}}\right\|_{L^2(\Omega)}^{\frac{2(p+2l-m-3)}{p+m-1}\alpha}
\left\|n_\varepsilon^{\frac{m+p-1}{2}}\right\|_{L^{\frac{2}{(p+m-1)}}(\Omega)}^{\frac{2(p+2l-m-3)}{p+m-1}(1-\alpha)}\notag\\
&+\left\|n_\varepsilon^{\frac{m+p-1}{2}}\right\|_{L^{\frac{2}{(p+m-1)}}(\Omega)}^{\frac{2(p+2l-m-3)}{p+m-1}}\notag\\
\leq& C_8\left\|\nabla n_\varepsilon^{\frac{m+p-1}{2}}\right\|_{L^2(\Omega)}^{\frac{2(p+2l-m-3)}{p+m-1}\alpha}+C_8\ \ \mathrm{for\ all}\ t\in(0,T_{max,\varepsilon}),
\end{align}
while $\alpha\in(0,1)$ is determined by
\begin{align}\label{alpha dereremine identity}
\frac{(p+m-1)q}{2(p+2l-m-3)(q+1)}=\left(\frac{1}{2}-\frac{1}{3}\right)\cdot\alpha+\frac{2}{m+p-1}(1-\alpha),
\end{align}
and
\begin{align}\label{condition 1 for G-N}
\frac{2(p+2l-m-3)(q+1)}{(m+p-1)q}<6
\end{align}
is needed due to the embedding $W^{1,2}(\Omega)\hookrightarrow L^{6}(\Omega)$. Because of $m>l-1$ and $p>m-2l+3$, we can see that
\begin{align*}
&6(m+p-1)q-2(p-m+2l-3)(q+1)\\
=&2(p-m+2l-3)(2q-1)+12(m-l+1)q>0,
\end{align*}
which indicates
\begin{align*}
\frac{2(p+2l-m-3)(q+1)}{(m+p-1)q}-6<0,
\end{align*}
then (\ref{condition 1 for G-N}) is satisfied. According to the identity (\ref{alpha dereremine identity}), we have
\begin{align*}
\alpha :=\frac{3(m+p-1)}{3m+3p-4}\left(1-\frac{q}{(p+2l-m-3)(q+1)}\right)\in(0,1).
\end{align*}
Then, by inserting $\alpha$ into (\ref{estimate 5 the first term}) we see that
\begin{align}\label{estimate 5 details}
\left\|n_\varepsilon^{\frac{m+p-1}{2}}\right\|_{L^{\frac{2(p+2l-m-3)(q+1)}{(p+m-1)q}}(\Omega)}^{\frac{2(p+2l-m-3)}{p+m-1}}\leq C_9\left(\int_\Omega\left|\nabla n_\varepsilon^{\frac{m+p-1}{2}}\right|^{2}{\rm d}x+1\right)^{\frac{3[(p-m+2l-3)(q+1)-q]}{(3m+3p-4)(q+1)}}
\end{align}
for some $C_9=C(p,q)>0$.

Next, we apply the Lemma \ref{interpolation lemma about varphi} to estimate $\|\nabla c_\varepsilon\|_{L^{2q+2}(\Omega)}^{2}$ as follows:
\begin{align}\label{estimate 5 the second term}
\|\nabla c_\varepsilon\|_{L^{2q+2}(\Omega)}^{2}&\leq C_{10}\||\nabla c_\varepsilon|^{q-1}D^2c_\varepsilon\|_{L^2(\Omega)}^{\frac{2}{q+1}}\|c_\varepsilon\|_{L^\infty(\Omega)}^{\frac{2}{q+1}}
+\|c_\varepsilon\|_{L^\infty(\Omega)}^{2}\notag\\
&\leq C_{11}\left(\int_\Omega|\nabla c_\varepsilon|^{2q-2}|D^2c_\varepsilon|^2{\rm d}x+1\right)^{\frac{1}{q+1}}\ \ \ \mathrm{for\ all}\ t\in(0,T_{max,\varepsilon}).
\end{align}
for some positive constants $C_{10}=C(p,q)$ and $C_{11}=C(p,q)$ by choosing $\lambda=2q+2$ in (\ref{interpolation inequality about varphi}).

Thus, combining with (\ref{estimate 5}), (\ref{estimate 5 the first term}) and (\ref{estimate 5 the second term}) and employing the Young inequality, we can find  that
\begin{align}\label{estimate 3 summary 1}
\int_\Omega(n_\varepsilon+\varepsilon)^{p+2l-m-3}|\nabla c_\varepsilon|^2{\rm d}x\leq &C_{10}C_{11}\left(\int_\Omega\left|\nabla n_\varepsilon^{\frac{m+p-1}{2}}\right|^{2}{\rm d}x+1\right)^{\frac{3[(p-m+2l-3)(q+1)-q]}{(3m+3p-4)(q+1)}}\times\notag\\&\quad\quad\left(\int_\Omega|\nabla c_\varepsilon|^{2q-2}|D^2c_\varepsilon|^2{\rm d}x+1\right)^{\frac{1}{q+1}}\notag\\
\leq& \eta\left(\int_\Omega|\nabla c_\varepsilon|^{2q-2}|D^2c_\varepsilon|^2{\rm d}x+1\right)+\notag\\
&\quad\quad C_{12}\left(\int_\Omega\left|\nabla n_\varepsilon^{\frac{m+p-1}{2}}\right|^{2}{\rm d}x+1\right)^{\frac{3[(p-m+2l-3)(q+1)-q]}{q(3m+3p-4)}}
\end{align}
for all $t\in(0,T_{max,\varepsilon})$ and $\eta>0$, where the positive constant $C_6$ is related to $\eta$, $p$ and $q$. Otherwise, by (\ref{upper bound of p}), we can obtain that
\begin{align*}
&3\left[(p+2l-m-3)(q+1)-q\right]-q(3m+3p-4)\\
=&3[p-(2m-2l+\frac{8}{3})q+(p+2l-3)]>0,
\end{align*}
which implies that
$$\frac{3[(p-m+2l-3)(q+1)-q]}{q(3m+3p-4)}<1.$$
Thus, we can employ the Young inequality again to the second term on right-hand side in (\ref{estimate 3 summary 1}) to finished the proof of this lemma finally.
\end{proof}
\vskip 3mm

We are in position to estimate the terms on the right-hand side in (\ref{differential inequality about nabla c 2q}). The following three lemmas are cited from  lemma 3.10, lemma 3..11 and lemma 3.12 in \cite{MR3426095} respectively, the proof details of which are omitted here.
\begin{lem}\label{lem:below bounds of p}
Let $m\geq 1$ and $q>1$. Assume that $p>1$ satisfies
\begin{align}\label{below bound of p}
p>\frac{3q-3m+4}{3}.
\end{align}
Then, for each $\eta>0$ and a positive constant $C_{13}=C(p,q,\eta)$, the classical solution to the system (\ref{the regularized system}) have the property
\begin{align}\label{estimate 6}
\int_\Omega n_\varepsilon^2|\nabla c_\varepsilon|^{2q-2}{\rm d}x\leq\eta\int_\Omega\left|\nabla n_\varepsilon^{\frac{m+p-1}{2}}\right|^2{\rm d}x +\eta\int_\Omega\left|\nabla c_\varepsilon|^{2q-2}\right|D^2c_\varepsilon|^2{\rm d}x+C_{13}\ \ \ \mathrm{for\ all}\ t\in(0,T_{max,\varepsilon}).
\end{align}

\end{lem}
\vskip 2mm
\begin{lem}\label{lem:u varepsilon condition 1}
Let $m\geq 1$ and $r>\frac{3}{2}$ and suppose that $q\geq r-1$ is such that
\begin{align}\label{q and r relation 1}
(4-2r)q\leq r-1.
\end{align}
Then, for all $\eta>0$ and each $K>0$, there exists $C_{14}=C(q,r,\eta,K)>0$ such that if for some $\varepsilon\in(0,1)$ we have
\begin{align}\label{u varepsilon condition 1}
\|Du_\varepsilon(\cdot,t)\|_{L^r(\Omega)}\leq K\ \ \ \ \mathrm{for\ all}\ t\in(0,T_{max,\varepsilon}),
\end{align}
then
\begin{align}\label{estimate 7 in 1 condition}
\int_\Omega|\nabla c_\varepsilon|^{2q}|Du_\varepsilon|{\rm d}x\leq\eta\int_\Omega|\nabla c_\varepsilon|^{2q-2}|D^2c_\varepsilon|^2{\rm d}x +C_{14}\ \ \ \mathrm{for\ all}\ t\in(0,T_{max,\varepsilon}).
\end{align}
\end{lem}
\vskip 2mm
\begin{lem}\label{lem:u varepsilon condition 2}
Let $m\geq 1$, and suppose that $r\in(1,\frac{3}{2}]$ and
\begin{align}\label{q and r relation 2}
q\in\left(1,\frac{2r+3}{3}\right).
\end{align}
Then for each $\eta>0$ and $K>0$ one can find $C_{15}=C(q,r,\eta,K)>0$ such that if there exist $\varepsilon\in(0,1)$ fulfilling
\begin{align}\label{u varepsilon condition 2}
\|u_\varepsilon\|_{L^r(\Omega)}\leq K\ \ \ \mathrm{for\ all}\ t\in(0,T_{max,\varepsilon}),
\end{align}
then
\begin{align}\label{estimate 7 in 2 condition}
\int_\Omega|\nabla c_\varepsilon|^{2q}|u_\varepsilon|^2{\rm d}x\leq\eta\int_\Omega|\nabla c_\varepsilon|^{2q-2}|D^2c_\varepsilon|^2{\rm d}x+\eta\int_\Omega|Au_\varepsilon|^2{\rm d}x +C_{15}\ \ \ \mathrm{for\ all}\ t\in(0,T_{max,\varepsilon}).
\end{align}
\end{lem}
\vskip 3mm
\section{Combining previous estimates}
\vskip 2mm
Now if $m>l-1$, then the conditions on $p$ in Lemmas \ref{lem:upper bounds p} and \ref{lem:below bounds of p} can be fulfilled simultaneously for any choice of $q>1$ and $l>2$. Thus, resorting to such $m$ allows for combining the above results to derive an ODI. And we note that all constants appear in this section is independent of $\varepsilon$.
\begin{lem}\label{lem:summary 1}
Assume that $m>l-1$. Let $r\geq 1$ and $q>1$ satisfy
\begin{eqnarray}\label{q and r relation summary}
  \left\{\begin{array}{lll}
  \medskip
  q<\frac{3+2r}{3},\ \ \ &\mathrm{if}\ r\leq\frac{3}{2},\\
  \medskip
  (4-2r)q\leq r-1,\ \ \ &\mathrm{if}\ r>\frac{3}{2},
\end{array}\right.
\end{eqnarray}
and assume that $p>\max\{l-1,m-2l+3\}$ be such that
\begin{align}\label{p q m condition}
\frac{3q-3m+4}{3}<p<\left(2m-2l+\frac{8}{3}\right)q+m-2l+3.
\end{align}
Then for all $K>0$ one can find a constant $C_{16}=C(p,q,r,K)>0$ such that if for some $\varepsilon\in(0,1)$ and $T_{max,\varepsilon}>0$ we have
\begin{align}\label{u varepsilon condition summary}
\|Du_\varepsilon(\cdot,t)\|_{L^r(\Omega)}\leq K\ \ \ \mathrm{for\ all}\ t\in(0,T_{max,\varepsilon}),
\end{align}
then
\begin{align}\label{odi inequality}
&\frac{{\rm d}}{{\rm d}t}\left\{\int_\Omega(n_\varepsilon+\varepsilon)^p{\rm d}x+\int_\Omega|\nabla c_\varepsilon|^{2q}{\rm d}x+\int_\Omega|A^{\frac{1}{2}}u_\varepsilon|^2{\rm d}x\right\}
+\frac{1}{C_{16}}\cdot\left\{\int_\Omega|\nabla n_\varepsilon^{\frac{m+p-1}{2}}|^2{\rm d}x\right.\notag\\ &\left.+\int_\Omega|\nabla c_\varepsilon|^{2q-2}|D^2c_\varepsilon|^2{\rm d}x+\int_\Omega|Au_\varepsilon|^2{\rm d}x\right\}
\leq C_{16}.\ \ \ \mathrm{for\ all}\ t\in(0,T).
\end{align}
\end{lem}
\begin{proof}
To obtain the ODI inequality (\ref{odi inequality}), we only need to combine Lemma \ref{boundedness of n varepsilon in L^p}-\ref{lem:estimate to u varepsilon} with Lemma \ref{lem:upper bounds p}-\ref{lem:u varepsilon condition 2} by choosing a suitable $\eta>0$.
\end{proof}
\vskip 2mm

Assuming the boundedness property of $u_\varepsilon$, upon a further analysis of (\ref{odi inequality}) we can estimate $n_\varepsilon$ in $L^\infty((0,\infty);L^{p}(\Omega))$ for certain $p\in(1,\infty)$.
\vskip 2mm
\begin{lem}\label{lem:L^p boundedness}
Let $m>l-1$, and assume that $r\geq1$ and $p>\max\{l-1,m-2l+3\}$ are such that there exists $q>1$ for which
(\ref{q and r relation summary}) and (\ref{p q m condition}) hold. Then for all $K>0$ there exists $C_{17}=C(p,q,r,K)>0$ with the property that if $\varepsilon\in(0,1)$ is such that
\begin{align}\label{u condition summary}
\|Du_\varepsilon\|_{L^{r}(\Omega)}\leq K\ \ \ \ \mathrm{for\ all}\ t\in(0,T_{max,\varepsilon}),
\end{align}
then we have
\begin{align}\label{n bounds under u varepsilon}
\int_\Omega n_\varepsilon^p(\cdot,t){\rm d}x\leq C_{17}\ \ \ \ \mathrm{for\ all}\ t\in(0,T_{max,\varepsilon}).
\end{align}
\end{lem}
\begin{proof}
In this lemma, we can derive the consequence from (\ref{odi inequality}) which is almost the same to the lemma 3.14 in \cite{MR3426095}, then we omit the details here.
\end{proof}

Now by virtue of the mass conservation of $n_\varepsilon$ (\ref{mass conservation of n varepsilon}), a first application of Lemma \ref{relationships between boundedness of n varepsilon and Du varepsilon} warrants that the assumption (\ref{u condition summary}) in the above lemma is fulfilled for some suitably small $r>1$. Adjusting the parameter $q$ properly, we thereby arrive at the following result which may be viewed as an improvement of the regularity property implied by (\ref{mass conservation of n varepsilon}) because the $5m^\star(l)-6l+\frac{25}{3}>1$ is fulfilled for any $l>2$.

\begin{lem}\label{lem:q upper bound related to m}
Let $m>m^\star(l)$ with $m^\star(l)$ defined as in (\ref{definition of m star}). Then for all $p>\max\{l-1,m-2l+3\}$ satisfying
\begin{align}\label{q upper bound related to m}
p<5m-6l+\frac{25}{3},
\end{align}
one can find $C_{18}=C(p)>0$ such that whenever $\varepsilon\in(0,1)$, we have
\begin{align}\label{n L^p bounds 1}
\|n_\varepsilon(\cdot,t)\|_{L^p(\Omega)}\leq C_{18}\ \ \ \ \mathrm{for\ all}\ t\in(0,T_{max,\varepsilon}).
\end{align}
\end{lem}
\begin{proof}
We first observe that $m^\star(l)\geq l-1$ and
\begin{align}\label{m and l relation}
&\max\{l-1,m-2l+3\}<p<5m-6l+\frac{25}{3},\\
&\frac{7}{3}-m<5m-6l+\frac{25}{3}
\end{align}
are satisfied with $m>m^\star(l)$ for all $l>2$.

Now since $p<5m-6l+\frac{25}{3}$ by the (\ref{q upper bound related to m}), we have
\begin{align}\label{q upper condition 1}
\frac{3(p+2l-m-3)}{6m-6l+8}<2.
\end{align}
Moreover,  the function
\begin{align*}
g(m)=18m^2-18lm+9m+6l-5
\end{align*}
is increasing when $m>\frac{l}{2}-\frac{1}{4}$ by differentiating the function above.
Therefore, the assumption $m>l-\frac{5}{6}$ ensures that for all $l>2$
\begin{align}\label{m l condition to q}
-18(m-l+\frac{5}{6})p\leq 18m^2-18lm+9m+6l-5
\end{align}
due to $g(l-\frac{5}{6})=0$ since $l-\frac{5}{6}>\frac{l}{2}-\frac{1}{4}$ is fulfilled, which is equivalent to the inequality
\begin{align}\label{q relationship with p m}
\frac{3(p-m+2l-3)}{6m-6l+8}<\frac{3p+3m-4}{3},
\end{align}
then the above inequality is also satisfied with $m$ chosen above.

According to (\ref{m l condition to q}), (\ref{q relationship with p m}) and the fact that
\begin{align}\label{below bounds}
\frac{3p+3m-4}{3}>1
\end{align}
by $p>\frac{7}{3}-m$, we can fix $q\in(1,2)$ fulfilling
\begin{align}\label{upper and below bounds of q}
\frac{3(p+2l-m-1)}{6m-6l+8}<q<\frac{3p+3m-4}{3},
\end{align}
where the left inequality asserts that
$$p<5m-6l+\frac{25}{3},$$
and the right inequality guarantees that
$$p>\frac{3q-3m+4}{3},$$
altogether meaning that (\ref{p q m condition}) is satisfied. Due to $q<2$, we can finally pick $r\in[1,\frac{3}{2})$ sufficiently close to $\frac{3}{2}$ such that
$r>\frac{3q-3}{2}$, so that
$$q<\frac{2r+3}{3},$$
ensuring that also (\ref{q and r relation summary}) is valid. Then in view of (\ref{mass conservation of n varepsilon}), Lemma \ref{relationships between boundedness of n varepsilon and Du varepsilon} assets that
$$\|Du_\varepsilon\|_{L^r(\Omega)}\leq C_{19}\ \ \ \mathrm{for\ all}\ t\in(0,T_{max,\varepsilon})$$
with some $C_{19}>0$, whence according to the choices of $r$, $q$, and $p$, we may apply Lemma \ref{lem:summary 1} to find a $C_{18}=C(p)$ such that (\ref{n L^p bounds 1}) is satisfied. This proves the Lemma.
\end{proof}
\vskip 2mm

In a second step, on the basis of the knowledge just gained, we may apply the Lemma \ref{relationships between boundedness of n varepsilon and Du varepsilon} and once more combine the outcome thereof with Lemma \ref{lem:summary 1} to obtain bounds for $n_\varepsilon$ in any space $L^\infty((0,\infty);L^p(\Omega))$.
\vskip 2mm
\begin{lem}\label{lem:any p boundedness for n varepsilon}
Let $m>m^\star(l)$ with $m^\star(l)$ denoted in (\ref{definition of m star}). Then, for all $p>1$, there exists $C_{20}=C(p)>0$ such that for each $\varepsilon\in(0,1)$ we have
\begin{align}\label{n L^p boundedness 2}
\int_\Omega n_\varepsilon^p(\cdot,t){\rm d}x\leq C_{20}\ \ \ \mathrm{for\ all}\ t\in(0,T_{max,\varepsilon}).
\end{align}
\end{lem}
\begin{proof}
In this lemma, we only need to prove that there definitely exist some $p>p_0$ satisfying (\ref{n L^p boundedness 2}) for any $p_0>\max\{l-1,m-2l+3\}$ with some positive constant $C$.\\
For this purpose, given such $p_0$ we first fix $q>1$ satisfying
\begin{align}\label{realtion of p and q0}
q>\frac{3(p_0+2l-m-3)}{6m-6l+8}
\end{align}
and observe that then since $m>l-\frac{5}{6}$ we have
$$3q-3m+4-(6m-6l+8)q-3m+6l-9=(-6m+6l-5)q+(-6m+6l-5)<0$$
and hence
$$\frac{3q-3m+4}{3}<\left(2m-2l+\frac{8}{3}\right)q+m-2l+3.$$
As (\ref{realtion of p and q0}) ensures that moreover
$$\left(2m-2l+\frac{8}{3}\right)q+m-2l+3>\frac{6m-6l+8}{3}\cdot\frac{3(p_0+2l-m-3)}{6m-6l+8}++m-2l+3=p_0,$$
we can therefore pick some $p>p_0$ fulling
\begin{align}\label{upper and below bounds to p related to p0}
\frac{3q-3m+4}{3}<p<p_0<\left(2m-2l+\frac{8}{3}\right)q+m-2l+3.
\end{align}
Now in order to verify (\ref{n L^p boundedness 2}) for these choices of $p$ and $q$, we fist use use the fact that
$$5m-6l+\frac{25}{3}>\frac{19}{12}$$
for all $l>2$ under the assumption that $m>\max\left\{l-\frac{5}{6},\frac{7}{5}l-\frac{28}{15}\right\}$. Then, we can infer from the Lemma \ref{lem:q upper bound related to m} that there exists some $C_{21}>0$ fulfilling
$$\|n_\varepsilon(\cdot,t)\|_{L^{\frac{19}{12}}(\Omega)}\leq C_{21}\ \ \ \ \mathrm{for\ all}\ t\in(0,T_{max,\varepsilon}).$$
Since $\frac{3\frac{19}{12}}{3-\frac{19}{12}}=\frac{57}{17}>3$, Lemma \ref{relationships between boundedness of n varepsilon and Du varepsilon} yields $C_{22}>0$ satisfying
$$\|u_\varepsilon(\cdot,t)\|_{L^3(\Omega)}\leq C_{22}$$
for any $t\in(0,T_{max,\varepsilon})$, which also implies that the condition $(4-2r)q\leq r-1$ in (\ref{q and r relation summary}) is trivially satisfied, thanks to (\ref{upper and below bounds to p related to p0}) we may invoke Lemma \ref{lem:L^p boundedness} to establish (\ref{n L^p boundedness 2}).
\end{proof}

By the application to some general semigroup estimates and the standard parabolic regularity arguments, we can derive that the classical solution to the system (\ref{the regularized system}) is global in time, at the same time some boundeness results can be obtained.
\vskip 3mm
\begin{lem}\label{lem:n,c,u boundedness}
Let $m>m^\star(l)$ with the definition of $m^\star(l)$ in (\ref{definition of m star}). Then the solution to the (\ref{the regularized system}) is global in time for all $\varepsilon\in(0,1)$ and also bounded as follows:
\begin{align}\label{n boundedess in L^infty}
\|n_\varepsilon(\cdot,t)\|_{L^\infty_{{\rm loc}}(0, \infty; L^\infty(\Omega))}\leq C_{23}
\end{align}
and
\begin{align}\label{c boundedess in W^1infty}
\|c_\varepsilon(\cdot,t)\|_{L^\infty_{{\rm loc}}(0, \infty;W^{1,\infty}(\Omega))}\leq C_{23}
\end{align}
as well as
\begin{align}\label{u boundedess in W^1infty}
\|u_\varepsilon(\cdot,t)\|_{L^\infty_{{\rm loc}}(0, \infty;W^{1,\infty}(\Omega))}\leq C_{23}.
\end{align}
\end{lem}
\begin{proof}
First,the validity of estimate (\ref{n L^p boundedness 2}) for any $p>3$ allows for an application of Lemma \ref{relationships between boundedness of n varepsilon and Du varepsilon} to $r=\infty$ to infer that $Du_\varepsilon$ is bounded in $L^\infty(\Omega\times(0, T_{max,\varepsilon}))$, then we can reach that
\begin{align}\label{u local boundedess in W^1infty}
\|u_{\varepsilon}(\cdot,t)\|_{W^{1,\infty}(\Omega)}\leq C_{24}\ \ \ \ \mathrm{for\ all}\ t\in(0,T_{max,\varepsilon}).
\end{align}
Next, we establish the boundedness of $\|n_\varepsilon\|_{L^\infty(\Omega\times(0,T_{max,\varepsilon}))}$ for all $\varepsilon\in(0,1)$. According to the well-known estimate for the Neumann heat semigroup in $\Omega$, we can invoke the variation-of-constants formula for $n_\varepsilon$ and $\nabla\cdot u_\varepsilon=0$ to find that there exists a constant $C_{25}>0$ fulfilling
\begin{align}\label{eq:boundedness n varepsilon 1}
\|n_\varepsilon(\cdot,t)\|_{L^\infty(\Omega)}&\leq\|e^{t\Delta}n_0\|_{L^\infty(\Omega)}+\int_{0}^{t}\|e^{(t-s)\Delta}\nabla(n_\varepsilon S(x,n_\varepsilon,c_\varepsilon)\nabla c_\varepsilon +n_\varepsilon u_\varepsilon)(\cdot,s)\|_{L^\infty(\Omega)}{\rm d}s\\
&\leq \|n_0\|_{L^\infty(\Omega)}+C_{25}\int_{0}^{t}(t-s)^{(-\frac{1}{2}-\frac{3}{2q})}e^{-\lambda_1(t-s)}\|n_\varepsilon S(x,n_\varepsilon,c_\varepsilon)\nabla c_\varepsilon +n_\varepsilon u_\varepsilon\|_{L^q(\Omega)}{\rm d}s.
\end{align}
Thus, for any given $3<q<r$, we can find that
\begin{align}\label{eq:boundedness n varepsilon 2}
\|n_\varepsilon S(x,n_\varepsilon,c_\varepsilon)\nabla c_\varepsilon\|_{L^q(\Omega)}&\leq C_{0}\|n_{\varepsilon}^{l-1}\nabla
c_\varepsilon\|_{L^q(\Omega)}
\notag\\
&\leq C_0\|n_\varepsilon(\cdot,t)\|_{L^r(l-1)(\Omega)}^{l-1}\|\nabla c_\varepsilon(\cdot,t)\|_{L^{\frac{qr}{r-q}}(\Omega)}\leq C_{26},
\end{align}
and
\begin{align}\label{eq:boundedness n varepsilon 3}
\|n_\varepsilon u_\varepsilon\|_{L^q(\Omega)}\leq\|u_\varepsilon(\cdot,t)\|_{L^\infty(\Omega)}\|n_{\varepsilon}(\cdot,t)\|_{L^q(\Omega)}\leq C_{27}
\end{align}
for all $t\in(0,T_{max,\varepsilon})$ by using the $\mathrm{H\ddot{o}lder}$ inequality, Lemma \ref{lem:any p boundedness for n varepsilon} and the  boundedness result for $u_\varepsilon$. Then, we can infer from (\ref{eq:boundedness n varepsilon 1})-(\ref{eq:boundedness n varepsilon 3})
that
\begin{align}\label{condition 1}
\|n_{\varepsilon}(\cdot,t)\|_{L^\infty(\Omega)}\leq C_{26},\ \ \ \mathrm{for\ all}\ t\in(0,T_{max,\varepsilon})
\end{align}
is true
because $q>3$ ensures that $\int_0^\infty (t-s)^{(-\frac{1}{2}-\frac{3}{2q})}e^{-\lambda_1(t-s)}{\rm d}s$ is finite.

Moreover, it can be derived from an well-known arguments parabolic regularity theory that
\begin{align}\label{condition 2}
\|c_\varepsilon(\cdot,t)\|_{W^{1,\infty}(\Omega)}\leq C_{27}\ \ \ \ \mathrm{for\ all}\ t\in(0,T_{max,\varepsilon})
\end{align}
is satisfied due to (\ref{condition 1}).\\
We now turn to estimate for $A^\alpha u_\varepsilon$. Applying the fractional power $A^\alpha$ ($\alpha$ is given in (\ref{the initial data conditions})) to the variation-of-constants formula
$$u_\varepsilon(\cdot,t)=e^{-tA}u_{0}+\int_{0}^{t}e^{(t-s)A}\mathscr{P}(n_\varepsilon(\cdot,s)\nabla\varphi){\rm d}s,\ \ \ \mathrm{for\ all}\ t\in(0.T_{max,\varepsilon}),$$
we can arrive at
\begin{align}\label{condition 3 for global existence}
\|A^\alpha u_\varepsilon(\cdot,t)\|_{L^2(\Omega)}\leq\|A^\alpha e^{-tA}u_{0}\|_{L^2(\Omega)}+\int_{0}^{t}\|A^\alpha e^{(t-s)A}\mathscr{P}(n_\varepsilon(\cdot,s)\nabla\varphi)\|_{L^2(\Omega)}{\rm d}s
\end{align}
For the first term on the right-hand side, we can easily obtain that
\begin{align}\label{condition 3 for global existence 1}
\|A^\alpha e^{-tA}u_{0}\|_{L^2(\Omega)}=\|e^{-tA}A^\alpha u_{0}\|_{L^2(\Omega)}\leq C_{29}
\end{align}
is valid for some positive constant $C_{18}$ and all $t\in(0,T_{max,\varepsilon})$ due to (\ref{the initial data conditions}).

Since the operator $\mathscr{P}$ is bounded from $L^2(\Omega)$ to $L^2_\sigma(\Omega)$, we can estimate the last term on the right-hand side as follows:
\begin{align}\label{condition 3 for global existence 2}
&\int_{0}^{t}\|A^\alpha e^{(t-s)A}\mathscr{P}(n_\varepsilon(\cdot,s)\nabla\varphi)\|_{L^2(\Omega)}{\rm d}s\notag\\ \leq
&C_{30}\|\varphi\|_{W^{1,\infty}(\Omega)}\|n_\varepsilon(\cdot,t)\|_{L^{\infty}(\Omega)}\int_0^\infty(t-s)^\alpha e^{-\lambda(t-s)}{\rm d}s\leq C_{31}
\end{align}
with some $C_{30}>0$ and $C_{31}>0$.

Substituting (\ref{condition 3 for global existence 1}) and (\ref{condition 3 for global existence 2}) into (\ref{condition 3 for global existence}),
we can conclude that there exists a $C_{21}>0$ such that
\begin{align}\label{condition 3}
\|A^\alpha u_\varepsilon(\cdot,t)\|_{L^{2}(\Omega)}\leq C_{32},\ \ \ \mathrm{for\ all}\ t\in(0,T_{max,\varepsilon}).
\end{align}
Then, combining (\ref{condition 1}), (\ref{condition 2}) and (\ref{condition 3}), we thus can claim that $T_{max,\varepsilon}=\infty$ and that the classical solution $(n_\varepsilon, c_\varepsilon, u_\varepsilon)$ is global in time. Therefore, we can futheremore find that the solution to the system (\ref{the regularized system}) satisfies (\ref{n boundedess in L^infty})-(\ref{u boundedess in W^1infty}) by arguing similarly as above.
\end{proof}
\vskip 3mm
As one further class of a priori estimates, let us finally also note straightforward consequence of Lemma \ref{lem:n,c,u boundedness} for uniform $\mathrm{H\ddot{o}lder}$ regularity properties of $c_\varepsilon$, $\nabla c_\varepsilon$ and $u_\varepsilon$.
\begin{lem}
Let $m$ be denoted as the above lemma. Then there exists $\theta\in(0,1)$ such that for some $C>0$ we have
\begin{align}\label{holder regularity for c varepsilon}
\|c_\varepsilon\|_{C^{\theta,\frac{\theta}{2}}(\overline(\Omega)\times[t,t+1])}\leq C\ \ \ \mathrm{for\ all}\ t\geq 0,
\end{align}
and
\begin{align}\label{holder regularity for nabla c varepsilon}
\|\nabla c_\varepsilon\|_{C^{\theta,\frac{\theta}{2}}(\overline(\Omega)\times[t,t+1])}\leq C(\tau)\ \ \ \mathrm{for\ all}\ t\geq\tau,
\end{align}
as well as
\begin{align}\label{holder regularity for u varepsilon}
\|u_\varepsilon\|_{C^{\theta,\frac{\theta}{2}}(\overline(\Omega)\times[t,t+1])}\leq C\ \ \ \mathrm{for\ all}\ t\geq 0,
\end{align}
with any $\tau\in(0,\infty)$.
\end{lem}
\begin{proof}
The proof of this lemma is  on the standard parabolic regularity theory and some standard semigroup estimation techniques, which is omitted here. For the details, we can refer to the lemma 3.18 and lemma 3.19 in \cite{MR3426095}.
\end{proof}
\vskip 3mm
\section{proof of the theorem \ref{thm:the main theorem}}
In this section, we use the classical solution of the regularized system (\ref{the regularized system}) to approximate the weak solution we defined in Definition \ref{definition:definition of weak solution} above. At first, several necessary boundedness results are established in the following lemmas.
\vskip 2mm
\begin{lem}\label{lem:several boundedness estimates}
Let $m>m^\star(l)$ with $m^\star(l)$ given in (\ref{definition of m star}). Then, for each $\varepsilon\in(0,1)$, the global classical solution $(n_\varepsilon, c_\varepsilon, u_\varepsilon, p_\varepsilon)$  satisfies the following inequalities
\begin{align}\label{nf(c) bounds}
\int_0^\infty\int_\Omega n_\varepsilon f(c_\varepsilon){\rm d}x{\rm d}t\leq \int_\Omega c_0{\rm d}x,
\end{align}
and
\begin{align}\label{nabla c in L^2 bounds}
\int_0^\infty\int_\Omega|\nabla c_\varepsilon|^2{\rm d}x\leq\frac{1}{2}\int_\Omega c_0^2{\rm d}x,
\end{align}
as well as
\begin{align}\label{nabla n eta in L^2}
\int_0^\infty\int_\Omega|\nabla(n_\varepsilon+\varepsilon)^\eta|^2{\rm d}x{\rm d}\leq M_1,\ \ \ \ \mathrm{for\ any}\ \eta>\frac{m}{2}.
\end{align}
\end{lem}
\begin{proof}
The inequality (\ref{nf(c) bounds})-(\ref{nabla c in L^2 bounds}) can be obtained straightforwardly just by testing the second equation in (\ref{the regularized system}) with $1$ and $c_\varepsilon$ over $\Omega$ repectively, and using the nonnegativity of $n_\varepsilon$, $\nabla\cdot u_\varepsilon=0$.

We now establish the inequality (\ref{nabla n eta in L^2}). According to computation of the Lemma \ref{boundedness of n varepsilon in L^p} and the boundedness of $n_\varepsilon$ (\ref{n boundedess in L^infty}), we can obtain that
\begin{align*}
\frac{1}{p}\frac{{\rm d}}{{\rm d}t}\int_{\Omega}(n_\varepsilon+\varepsilon)^p{\rm d}x&+\frac{2m(p-1)}{(m+p-1)^2}\int_{\Omega}|\nabla (n_\varepsilon+\varepsilon)^{\frac{m+p-1}{2}}|^2{\rm d}x\\
&\leq\frac{C_{0}^{2}(p-1)}{2m}(\|n_\varepsilon\|_{L^{\infty}(\Omega)}+1)^{p+2l-m-3}\int_\Omega|\nabla c_{\varepsilon}|^{2}{\rm d}x
\end{align*}
for $p>1$. Then, an integration over $(0,t)$ yields that
\begin{align*}
\frac{1}{p}\int_{\Omega}(n_\varepsilon+\varepsilon)^p{\rm d}x+&\frac{2m(p-1)}{(m+p-1)^2}\int_0^t\int_{\Omega}|\nabla (n_\varepsilon+\varepsilon)^{\frac{m+p-1}{2}}|^2{\rm d}x\\
\leq&\frac{C_{0}^{2}(p-1)}{2m}(\|n_\varepsilon\|_{L^{\infty}(\Omega)}+1)^{p+2l-m-3}\int_0^t\int_\Omega|\nabla c_{\varepsilon}|^{2}{\rm d}x\\
&+\frac{1}{p}\int_\Omega(n_0+1)^p{\rm d}x\\
\leq& M_1
\end{align*}
for each $t\in(0,\infty)$.

For any $p>1$, we can infer from (\ref{nabla c in L^2 bounds}) that the (\ref{nabla n eta in L^2}) is satisfied by setting
$$\eta=\frac{m+p-1}{2},$$
which furthermore shows
$$\eta>\frac{m}{2}.$$
In particularly, we can obtain the boundedness of $\|\nabla(n_\varepsilon+\varepsilon)^m\|_{L^2_{\rm loc}(0,\infty; L^2(\Omega))}$ by choosing $\eta=m$.
\end{proof}
\vskip 3mm
\begin{lem}\label{lem:regularity 1 for partial n t}
Let $\varepsilon\in(0,1)$, $T\in(0,\infty)$ and $m$ be given as above lemmas. Suppose that $(n_\varepsilon, c_\varepsilon, u_\varepsilon, P_\varepsilon)$ is a classical solution to the regularized system (\ref{the regularized system}) on $[0,T)$. There exists a $\varepsilon$ independent constant $M_1(T)>0$ such that
\begin{align}\label{regularity 1 for partial n t}
\left\|\frac{\partial}{\partial t}(n_\varepsilon+\varepsilon)^\gamma\right\|_{L^1((0,T),(W^{3,2}_{0}(\Omega))^\star)}\leq M_2
\end{align}
for all $\gamma>\{1,\frac{m}{2}-l+2\}$.
\end{lem}
\begin{proof}
On differentiation and integration by parts in (\ref{the regularized system}), we see that for each $\varphi\in C^{\infty}_{0}(\Omega)$ we have
\begin{align}\label{estimate 1 for partial n t}
\frac{1}{\gamma}\int_\Omega(\frac{\partial}{\partial t}(n_\varepsilon+\varepsilon)^\gamma)\varphi{\rm d}x=&-\int_\Omega\nabla(n_\varepsilon+\varepsilon)^m\cdot\nabla((n_\varepsilon+\varepsilon)^{\gamma-1}\varphi){\rm d}x-\int_\Omega (u_\varepsilon\cdot\nabla n_\varepsilon)\cdot((n_\varepsilon+\varepsilon)^{\gamma-1}\varphi){\rm d}x\notag\\
&+\int_\Omega(n_\varepsilon S_\varepsilon(x,n_\varepsilon,c_\varepsilon)\nabla c_\varepsilon)\cdot\nabla((n_\varepsilon+\varepsilon)^{\gamma-1}\varphi){\rm d}x\notag\\
=& I_1+I_2+I_3,\ \ \ \mathrm{for\ all}\ t\in(0,T),
\end{align}
In order to estimate the $\{I_i\}_{i=1,2,3}$ above, we apply the $\mathrm{H\ddot{o}lder}$ inequality and Cauchy inequality to obtain that
\begin{align}\label{estimate 11 for partial n t}
I_1=&-m(\gamma-1)\int_\Omega(n_\varepsilon+\varepsilon)^{m+\gamma-3}|\nabla n_\varepsilon|^2\varphi{\rm d}x-m\int_\Omega(n_\varepsilon+\varepsilon)^{m+\gamma-2}\nabla n_\varepsilon\cdot\nabla\varphi{\rm d}x\notag\\
=&-\frac{4m(\gamma-1)}{(m+r-1)^2}\int_\Omega|\nabla(n_\varepsilon+\varepsilon)^{\frac{m+\gamma-1}{2}}|^2\varphi{\rm d}x-\frac{m}{m+\gamma-1}\int_\Omega\nabla(n_\varepsilon+\varepsilon)^{m+\gamma-1}\cdot\nabla\varphi{\rm d}x\notag\\
\leq&\frac{4m(\gamma-1)}{(m+r-1)^2}\|\varphi\|_{L^\infty(\Omega)}\int_\Omega|\nabla(n_\varepsilon+\varepsilon)^{\frac{m+\gamma-1}{2}}|^2{\rm d}x+\frac{m}{m+\gamma-1}\int_\Omega|\nabla(n_\varepsilon+\varepsilon)^{m+\gamma-1}\cdot\nabla\varphi|{\rm d}x\notag\\
\leq&\frac{4m(\gamma-1)}{(m+r-1)^2}\|\varphi\|_{L^\infty(\Omega)}\int_\Omega|\nabla(n_\varepsilon+\varepsilon)^{\frac{m+\gamma-1}{2}}|^2{\rm d}x+\|\nabla(n_\varepsilon+\varepsilon)^{m+\gamma-1}\|_{L^2(\Omega)}\|\nabla\varphi\|_{L^2(\Omega)}
\end{align}
and
\begin{align}\label{estimate 12 for partial n t}
I_2&=-\frac{1}{\gamma}\int_\Omega(u_\varepsilon\cdot\nabla(n_\varepsilon+\varepsilon)^\gamma)\cdot\varphi{\rm d}x=-\frac{1}{\gamma}\int_\Omega(n_\varepsilon+\varepsilon)^\gamma u_\varepsilon\cdot\nabla\varphi{\rm d}x\notag\\
&\leq\frac{1}{\gamma}(\|n_\varepsilon\|_{L^\infty(\Omega)}+1)^{\gamma}\|u_\varepsilon\|_{L^\infty(\Omega)}\|\nabla\varphi\|_{L^\infty(\Omega)}
\end{align}
as well as
\begin{align}\label{estimate 13 for partial n t}
I_{3}=&\int_\Omega\varphi(n_\varepsilon S_\varepsilon(x,n_\varepsilon,c_\varepsilon)\nabla c_\varepsilon)\cdot\nabla(n_\varepsilon+\varepsilon)^{\gamma-1}{\rm d}x+\int_\Omega(n_\varepsilon+\varepsilon)^{\gamma-1}(n_\varepsilon S_\varepsilon(x,n_\varepsilon,c_\varepsilon)\nabla c_\varepsilon)\cdot\nabla\varphi{\rm d}x\notag\\
\leq&C_0\|\nabla c_\varepsilon\|_{L^\infty(\Omega)}\int_\Omega|(n_\varepsilon+\varepsilon)^{l-1}\nabla(n_\varepsilon+\varepsilon)^{\gamma-1}\varphi|{\rm d}x+C_0\int_{\Omega}(n_\varepsilon+\varepsilon)^{l+\gamma-2}|\nabla c_\varepsilon\cdot\varphi|{\rm d}x\notag\\
\leq&\frac{(\gamma-1)C_0}{l+\gamma-1}\|\nabla c_\varepsilon\|_{L^\infty(\Omega)}\int_\Omega|\nabla(n_\varepsilon+\varepsilon)^{l+\gamma-2}||\varphi|{\rm d}x++C_0\int_{\Omega}(n_\varepsilon+\varepsilon)^{l+\gamma-2}|\nabla c_\varepsilon\cdot\varphi|{\rm d}x\notag\\
\leq& C_0\|\nabla c_\varepsilon\|_{L^\infty(\Omega)}\|\nabla(n_\varepsilon+\varepsilon)^{l+\gamma-2}\|_{L^2(\Omega)}\|\varphi\|_{L^2(\Omega)}\notag\\
&+C_0|\Omega|(\|n_\varepsilon\|_{L^\infty(\Omega)}+1)^{l+\gamma-2}\|\nabla c_\varepsilon\|_{L^\infty(\Omega)}\|\nabla\varphi\|_{L^\infty(\Omega)}.
\end{align}
Combining (\ref{estimate 11 for partial n t})-(\ref{estimate 13 for partial n t}) with (\ref{estimate 1 for partial n t} together, we can see that there exists some constant $C_{22}>0$ satisfying
\begin{align*}
\left|\int_\Omega(\frac{\partial}{\partial t}(n_\varepsilon+\varepsilon)^\gamma)\varphi{\rm d}x\right|
\leq&C_{33}\left( \|\nabla(n_\varepsilon+\varepsilon)^\frac{m+\gamma-1}{2}\|_{L^2(\Omega)}^{2}+\|\nabla(n_\varepsilon+\varepsilon)^{m+\gamma-1}\|_{L^2(\Omega)}\right.\\ &\left.+\|\nabla(n_\varepsilon+\varepsilon)^{l+\gamma-2}\|_{L^2(\Omega)}+1\right)\|\varphi\|_{W^{1,\infty}(\Omega)}.
\end{align*}
According to the Lemma \ref{lem:several boundedness estimates} and the embedding $W^{3,2}(\Omega)\hookrightarrow W^{1,\infty}(\Omega)$, we can claim that (\ref{regularity 1 for partial n t}) is true when the following relation is fulfilled:
\begin{eqnarray}
\left\{\begin{array}{lll}
  \medskip
  m+\gamma-1>\frac{m+\gamma-1}{2}\geq\eta,\label{eq:relation 1}\\
  \medskip
l+\gamma-2\geq\eta,\label{eq:relation 2}
\end{array}\right.
\end{eqnarray}
that is
\begin{align}\label{condition of gamma}
\gamma>\max\{1,\frac{m}{2}-l+2\}.
\end{align}
The proof of this lemma is finished.
\end{proof}
\vskip 3mm
\begin{lem}
Let $m$ be given as above lemmas, the there exists $C>0$ such that
\begin{align}\label{regularity 2 for partial n t}
\|\partial_tn_\varepsilon(\cdot,t)\|_{(W^{2,2}_{0}(\Omega))^\star}\leq C\ \ \ \ \mathrm{for\ all}\ t>0\ \mathrm{and}\ \varepsilon\in(0,1).
\end{align}
In particular,
\begin{align}\label{regularity 3 for partial n t}
\|n_\varepsilon(\cdot,t)-n_\varepsilon(\cdot,s)\|_{(W^{2,2}_{0}(\Omega))^\star}\leq C|t-s|\ \ \ \ \mathrm{for\ all}\ t>0\ \mathrm{and}\ \varepsilon\in(0,1).
\end{align}
\end{lem}
\begin{proof}
The proof of this lemma can be seen in \cite{MR3426095} lemma 3.23.
\end{proof}
\vskip 3mm
\begin{lem}
Let $m$ be given as in theorem \ref{thm:the main theorem}. Then there exists a subsequence $(\varepsilon_j)_{j\in\mathbb{N}}$ such that $\varepsilon_j\searrow 0$ as $j\rightarrow\infty$ and that
\begin{align}
&n_{\varepsilon_j}\rightarrow n\ \ \ \mathrm{a.e\ in}\ \Omega\times(0,\infty),\label{n a.e convergent}\\
&n_{\varepsilon_j}\stackrel{\ast}{\rightharpoonup} n\ \ \ \mathrm{in}\ L^\infty(\Omega\times(0,\infty)),\label{n weak convergent}\\
&n_{\varepsilon_j}\rightarrow n\ \ \ \mathrm{in}\ C_{\mathrm{loc}}([0,\infty);(W^{2,2}_{0}(\Omega))^\star),\label{n strong convergent}\\
&\nabla(n_{\varepsilon_j}+\varepsilon_j)^m\rightharpoonup \nabla n^m\ \ \ \mathrm{in}\ L^2_{\mathrm{loc}}((0,\infty);L^2(\Omega)),\label{nabla n m weak convergent}\\
&c_{\varepsilon_j}\rightarrow c\ \ \ \mathrm{in}\ C_{\mathrm{loc}}(\overline{\Omega}\times[0,\infty))),\ \ \ \mathrm{a.e\ in}\ \Omega\times(0,\infty)\label{c strong convergent}\\
&\nabla c_{\varepsilon_j}\rightarrow \nabla c\ \ \ \mathrm{in}\ C_{\mathrm{loc}}(\overline{\Omega}\times[0,\infty))),\label{nabla c strong convergent}\\
&\nabla c_{\varepsilon_j}\stackrel{\ast}{\rightharpoonup} \nabla c\ \ \ \mathrm{in}\ L^\infty(\Omega\times(0,\infty)),\label{nabla c weak convergent}\\
&u_{\varepsilon_j}\rightarrow u\ \ \ \mathrm{in}\ C_{\mathrm{loc}}(\overline{\Omega}\times[0,\infty))),\label{u strong convergent}\\
&Du_{\varepsilon_j}\stackrel{\ast}{\rightharpoonup} Du\ \ \ \mathrm{in}\ L^\infty(\Omega\times[0,\infty))),\label{Du strong convergent}
\end{align}
with some triple $(n,c,u)$ which is a global weak solution of (\ref{the main system}) in the sense of Definition \ref{definition:definition of weak solution}. Moreover, $n$ satisfies
\begin{align}\label{n weak topology}
n\in C^{0}_{\omega-\star}([0,\infty);L^\infty(\Omega))
\end{align}
and
\begin{align}\label{mass conversation}
\int_\Omega n(x,t){\rm d}x=\int_\Omega n_0(x){\rm d}x\ \ \ \mathrm{for\ all}\ t>0.
\end{align}
\end{lem}
\begin{proof}
The proof of this lemma is based on the Aubin-Lions lemma, $\mathrm{Arzel\grave{a}}$-Ascoli theorem and the boundedness results obtained in the above lemmas. The proof details of this lemma is almost the same to Lemma 4.2 in \cite{MR3426095} after some small modification, therefore it is omitted here.
\end{proof}
\vskip 3mm
\text{Proof of Theorem \ref{thm:the main theorem}}.
In this part, we shall prove the limits $(n,c,u)$ mention above is a weak solution of problem (\ref{the main system}). As usual, testing the first equation in (\ref{the regularized system}) by $\varphi\in C_0^\infty(\overline{\Omega}\times[0,\infty))$, we can see that
\begin{align}\label{weak solution 1 to n varepsilon}
-\int_0^\infty\int_\Omega n_\varepsilon\varphi_t{\rm d}x{\rm d}t=&\int_\Omega n_0\varphi(\cdot,0){\rm d}x-\int_0^\infty \nabla(n_\varepsilon+\varepsilon)^m\cdot\nabla\varphi{\rm d}x{\rm d}t+\int_0^\infty\int_\Omega n_\varepsilon u_\varepsilon\cdot\nabla\varphi{\rm d}x{\rm d}t\notag\\
+&\int_0^\infty\int_\Omega n_\varepsilon S_\varepsilon(x,n_\varepsilon,c_\varepsilon)\nabla c_\varepsilon\cdot\nabla\varphi{\rm d}x{\rm d}t\ \ \ \mathrm{for\ all}\ \varepsilon\in(0,1).
\end{align}

Thus, by (\ref{n a.e convergent})-(\ref{c strong convergent})£¬(\ref{u strong convergent}) and the definition of matrix-valued function $S_\varepsilon$, the (\ref{identity 1 in weak sense}) can be obtained by passing to the limit in each term of the identity above. Along with a similar procedure applied to the second and the third equation in the system (\ref{the regularized system}), we can also deduce (\ref{identity 2 in weak sense}) and (\ref{identity 3 in weak sense}).

For the boundedness results(\ref{boundedness result}) to the weak solution $(n,c,u)$, we can deduce from (\ref{n boundedess in L^infty})-(\ref{u boundedess in W^1infty}) that
\begin{align*}
\|n\|_{L^\infty((0,T);L^\infty(\Omega))}\leq\liminf_{j\rightarrow\infty}\|n_{\varepsilon_j}\|_{L^\infty((0,T);L^\infty(\Omega))}<\infty,
\end{align*}
and
\begin{align*}
\|c\|_{L^\infty((0,T);W^{1,\infty}(\Omega))}\leq\liminf_{j\rightarrow\infty}\|c_{\varepsilon_j}\|_{L^\infty((0,T);W^{1,\infty}(\Omega))}<\infty,
\end{align*}
as well as
\begin{align*}
\|cu\|_{L^\infty((0,T);W^{1,\infty}(\Omega))}\leq\liminf_{j\rightarrow\infty}\|u_{\varepsilon_j}\|_{L^\infty((0,T);W^{1,\infty}(\Omega))}<\infty
\end{align*}
for all $T<\infty$. This completes the proof the Theorem \ref{thm:the main theorem}.
$\hfill\Box$

\vskip 3mm


\end{document}